\documentclass[12pt]{amsart}
\usepackage{amsfonts}
\usepackage{amssymb}
\usepackage{amsmath}
\usepackage{amsthm}
\usepackage{enumerate}
\usepackage[all]{xy}
\usepackage{graphicx}
\usepackage{appendix}
\newtheorem{theorem}{Theorem}
\newtheorem{corollary}{Corollary}
\newtheorem{lemma}{Lemma}

\newtheorem{conjecture}{Conjecture}

\newtheorem{theo}{Theorem}
\newcounter{tmp}

\theoremstyle{definition}
\newtheorem{definition}{Definition}

\def\C{\mathbb{C}}

\def\N{\mathbb{N}}

\def\R{\mathbb{R}}

\begin{document}

\title[Interval Map]{Zero Entropy Interval Maps And MMLS-MMA Property} 

\author{Yunping JIANG}

\address{ Department of Mathematics, 
Queens College of the City University of New York,
Flushing, NY 11367-1597 and 
Department of Mathematics
Graduate School of the City University of New York
365 Fifth Avenue, New York, NY 10016}
 
\email[]{yunping.jiang@qc.cuny.edu}
 


\thanks{The author is partially supported by an award from NSF (grant number DMS-1747905), a collaboration grant from the Simons Foundation (grant number 523341), PSC-CUNY awards, and a grant from NSFC (grant number 11571122).}

\begin{abstract}
We prove that the flow generated by any interval map with zero topological entropy is minimally 
mean-attractable (MMA) and minimally mean-L-stable (MMLS). One of the consequences is that 
any oscillating sequence is linearly disjoint with all flows generated by interval maps with zero topological entropy. 
In particular, the M\"obius function is orthogonal to all flows generated by interval maps with zero topological entropy (Sarnak's conjecture for interval maps).
Another consequence is a non-trivial example of a flow having the discrete spectrum.
 \end{abstract}

\maketitle

\section{\bf Introduction}
 
Let $\C$ and $\R$ denote the complex plane and the real line. Let $\mathbb{N}=\{1, 2, 3, \cdots, n\,\cdots\}$ be the set of natural numbers. 
Suppose $X$ is a compact metric space with metric $d (\cdot, \cdot)$. Let $C(X, \C)$ be the space of all continuous functions $\varphi: X\to \C$. 
Consider a continuous map $f: X\to X$. Then $f$ generates a flow (i.e. dynamical system) 
$$
{\mathcal X}=\{ f^{n}: X\to X\}_{n\in \{0\}\cup \mathbb{N}}
$$ 
where $f^{n}$ means the $n^{th}$ iteration 
$$
\stackrel{\underbrace{f\circ f\circ \cdots \circ f}}{n}
$$ 
and $f^{0}$ means the idenity.

The M\"obius function is the arithmetic function $\mu (n): \mathbb{N}\to \{ -1, 0, 1\}$ defined as 
$$
\mu (n) =\left\{ \begin{array}{ll}
                        1 & \hbox{ if $n=1$};\\
                        (-1)^r & \hbox{ if $n=p_{1}\cdots p_{r}$ for $r$ distinct prime numbers $p_{i}$};\\
                        0 & \hbox{ if $p^{2}|n$ for a prime number $p$.}
                        \end{array}\right.
$$
Sarnak proposed a conjecture which makes a connection between number theory and ergodic theory. The conjecture can be stated as 

\medskip
\begin{conjecture}[Sarnak~\cite{S1,S2}]~\label{sc}
The M\"obius function is linearly disjoint from (i.e. orthogonal to) all flows with zero topological entropy.
\end{conjecture}

The linear disjointness can be defined more generally. 
Suppose ${\bf c}=(c_{n})_{n\in \N}$ is a sequence of complex numbers. 

\medskip
\begin{definition}[Linear Disjointness]
We say ${\bf c}$ is {\em linearly disjoint} from $\mathcal{X}$ if
\begin{equation}\label{disjointness}
   \lim_{N\to \infty} \frac{1}{N}\sum_{n=1}^N c_n \varphi (f^{n}x) =0
\end{equation}
for any $\varphi\in C(X, \C)$ and any $x\in X$.
\end{definition}

In order to understand Conjecture~\ref{sc} in ergodic theory, we gave the following definition in~\cite{FJ}. 
 
 \medskip
\begin{definition}[Oscillating Sequence]~\label{os}
The sequence ${\bf c}$ is said to be an oscillating sequence if there is a constant $\lambda >1$ such that 
\begin{equation}~\label{cn}
K=\lim_{N\to \infty} \frac{1}{N} \sum_{n=1}^{N} |c_{n}|^{\lambda} <\infty
\end{equation} and if 
\begin{equation}~\label{oseq}
\lim_{N\to \infty} \frac{1}{N} \sum_{n=1}^{N} c_{n}e^{2\pi i n t}=0, \;\; \forall \; 0\leq t<1.
\end{equation}
\end{definition}

An important example of an oscillating sequence is the M\"obius sequence ${\bf u}=(\mu (n))_{n\in \mathbb{N}}$  due to Davenport's theorem~\cite{Da}. 

Mean-L-stable flows are an important class of dynamical systems in ergodic theory. They are defined as follows.  
Suppose $E\subset \mathbb{N}$. The upper density $\overline{D}(E)$ of $E$ in $\mathbb{N}$ is  
$$
\overline{D}(E) =\limsup_{N\to \infty} \frac{\#(E\cap [1, N])}{N}.
$$

 \medskip
\begin{definition}[MLS]~\label{mma}
A flow $\mathcal{X}$ is said to be {\em mean-L-stable} (MLS)
if for every $\epsilon > 0$, there is a $\delta > 0$ such that
$d(x, y) <\delta$ implies $d(f^nx,f^ny) <\epsilon$ for all $n\in \mathbb{N}$ except
a subset $E$ of natural numbers with upper density $\overline{D}(E)$ less than $\epsilon$.
\end{definition}

An example of a MLS flow is an equicontinuous flow.

 \medskip
\begin{definition}[Equicontinuity]~\label{mma}
A flow $\mathcal{X}$ is said to be {\em equicontinuous}
if for every $\epsilon > 0$, there is a $\delta > 0$ such that
$d(x, y) <\delta$ implies $d(f^nx,f^ny) <\epsilon$ for all $n\in \mathbb{N}$.
\end{definition}

We says a closed subset $K\subset X$ is minimal (with respect to $f$) if $f(K)\subset K$ and if $\overline{\{f^{n}x\;|\; n\in \mathbb{N}\}}=K$ for every $x\in K$. 
We use ${\mathcal K}$ to denote the sub-flow 
$
\{ (f|K)^{n}: K\to K\}_{n\in \{0\}\cup  \mathbb{N}}.
$
Again, for the purpose of understanding Conjecture~\ref{sc} in ergodic theory, 
we gave the following definitions in~\cite{FJ}. 
 
 \medskip
\begin{definition}[MMLS]~\label{mmls}
We say that a flow $\mathcal{X}$ is {\em minimally MLS} (MMLS) if
for every minimal subset $K\subseteq X$, the sub-flow $\mathcal{K}$ is MLS.
\end{definition}
 
\medskip
\begin{definition}[MMA]~\label{mma}
Suppose $\mathcal{X}$ is a flow. Suppose $K\subset X$ is minimal.
We say $x\in X$ is {\em mean-attracted} to $K$ if for any $\epsilon >0$ there is a point $z=z_{\epsilon, x}\in K$
(depending on $x$ and $\epsilon$) such that
\begin{equation}\label{MA}
    \limsup_{N\to\infty} \frac{1}{N} \sum_{n=1}^{N} d(f^n x, f^n z) <\epsilon.
\end{equation}
The {\em basin of mean-attraction} of $K$, denoted $\hbox{\rm Basin}(K)$, is defined to be the set of all
points $x$ which are mean-attracted to $K$. It is trivial that
$K \subset \hbox{\rm Basin}(K)$.
We call $\mathcal{X}$ {\em minimally mean-attractable} (MMA) if
\begin{equation}\label{Decomp}
   X= \bigcup_{K} \mbox{\rm Basin}(K)
\end{equation}
where $K$ varies among all minimal subsets of $X$.
\end{definition}

Recall that a point $x\in X$ is attracted to $K$ if
$$
\lim_{n\to \infty} d(f^n x, K)=0.
$$
In general, this does not imply to that $x$ is mean-attracted to $K$.

We proved the following theorem in~\cite{FJ}.

\begingroup
\setcounter{tmp}{\value{theo}}
\setcounter{theo}{0} 
\renewcommand\thetheo{\Alph{theo}}

\medskip
\begin{theo}[MMA/MMLS and Disjointness]~\label{fjthm}
Any oscillating sequence ${\bf c}$ is linearly disjoint from all MMA and MMLS flows $\mathcal{X}$.
Moreover, the limit in (\ref{disjointness}) is uniform on every minimal subset $K$.
\end{theo}

\endgroup
\setcounter{theo}{\thetmp} 

Since a MMA and MMLS flow has zero topological entropy (for example, this follows from~\cite{LTY}),  
this theorem confirms Sarnak's conjecture for all MMA and MMLS flows.  We provided some examples 
of flows which are MMA and MMLS in~\cite{FJ}. However, these examples except 
for Denjoy counter-examples are equicontinuous when they are restricted on their minimal subsets. 
It is an interesting problem to find further examples which are MMA and MMLS but not equicontinuous 
when they are restricted on their minimal subsets. 
We give a complete answer to this problem for interval maps with zero topological entropy in this paper. 
The main result in this paper is that  

\medskip
\begin{theorem}[Main Theorem]~\label{mainthm}
Suppose $I=[a, b]$ is a closed interval. Suppose $f: I\to I$ is an interval map with zero topological entropy. Then the flow ${\mathcal X}$ generated by $f$ 
is a MMA and MMLS flow. 
\end{theorem}

One consequence of the main result (Theorem~\ref{mainthm}) with Theorem~\ref{fjthm} is

\medskip
\begin{corollary}[Interval Map and Disjointness]~\label{id}
Any oscillating sequence ${\bf c}$ is linearly disjoint from all flows ${\mathcal X}$ generated by interval maps $f: I\to I$ with zero topological entropy.
\end{corollary}

Corollary~\ref{id} and Davenport's theorem~\cite{Da} provides a completely new proof of Sarnak's conjecture for 
all flows generated by interval maps with zero topological entropy (see~\cite{K} for a different proof).  

\medskip
\begin{corollary}[Interval Map and M\"obius Function]~\label{im}
The M\"obius function is linearly disjoint from all flows ${\mathcal X}$ generated by interval maps $f: I\to I$ with zero topological entropy
\end{corollary}
 
Suppose $\mathcal{X}$ is a flow generated by $f$ and suppose $\mu$ is a regular Borel probability measure on $X$. 
Let $L^2(\mu)$ be the space of all $L^2$ functions. We say $\mu$ is $f$-invariant if 
$$
\mu (f^{-1}(A)) =\mu (A) \hbox{ for all Borel sets $A\subset X$}.
$$
Suppose $\mu$ is an $f$-invariant regular Borel probability measure. A complex number $\lambda$ is called an eigenvalue 
if there is a function $\phi \in L^2(\mu)$ such that $\phi( fx)=\lambda \phi (x)$ for $\mu$-a.e.$x\in X$. Here $\phi$ is called an eigenfunction with respect to $\lambda$.  
We say $(\mathcal{X}, \mu)$ has discrete spectrum or pure point spectrum if there exists an orthonormal $L^{2}(\mu)$-basis consisting of eigenfunctions. Theorem~\ref{mainthm} combining with~\cite[Theorem 3.8]{LTY} provides a non-trivial example of a MLS flow having discrete spectrum. 

\medskip
\begin{corollary}[Discrete Spectrum]~\label{dsp}
Suppose $\mathcal{X}$ is the flow generated by an interval map $f$. Suppose $K\subseteq X$ is a minimal subset of $f$. Then every ergodic invariant measure $\mu$ on $\mathcal{K}=(K, f)$ has discrete spectrum. 
\end{corollary} 

In additional, Corollary~\ref{dsp} combining with the work in~\cite{MRT,HWZ} gives another proof of Sarnak's conjecture for
all flows generated by interval maps with zero topological entropy as present in Corollary~\ref{im}.   
  
I would like to point out that there are flows generated by higher-dimensional maps with zero topological entropy which are not MLS when they are restricted on their minimal subsets. 
Thus in the higher-dimensional case, the oscillation property alone is not enough. We need instead the higher-order oscillation property. 
The reader who is interested in the higher-order oscillation property for a sequence ${\bf c}$ 
and the role it plays in Sarnak's conjecture for higher-dimensional dynamical systems is encouraged to read my paper~\cite{J}.   

We will give a proof of  Theorem~\ref{mainthm} in the next section.

\medskip
\medskip
\noindent {\em Acknowledgement.}  This paper was completed during my visit to the National Center for Theoretical Sciences (NCTS) 
at National Taiwan University.  I would like to thank NCTS for its hospitality and Professor Jung-Chao Ban for helpful discussions. I would like to thank John Adamski for his help on revising of the paper. 

\section{\bf Proof of Theorem~\ref{mainthm}} 
 
Suppose $I=[a, b]\subset {\mathbb R}$ is a closed interval and $f: I\to I$ is an interval map. 
Suppose the topological entropy $ent(f) =0$. In other words, this means that $f$ cannot have 
any horseshoe-type invariant subset in $I$. More precisely, this means that $f$ can not have 
a forward $f$-invariant closed subset $X\subset I$ 
such that $f: X\to X$ is topologically conjugate (or semi-conjugate) to a sub-shift of finite type 
$\sigma_{A}:\Sigma_{A}\to \Sigma_{A}$ with the maximal Lyapunov exponent greater than $0$.
For the sub-shift of finite type $\sigma_{A}:\Sigma_{A}\to \Sigma_{A}$, we mean that $A=(a_{ij})_{d\times d}$ is a $d\times d$-square matrix of $0$'s and $1$'s, 
$$
\Sigma_{A}= \{ w=i_{0}i_{1}\cdots i_{n-1}i_{n}\cdots \;|\; i_{n-1}\in \{1, 2, \cdots, d\},\; a_{i_{n-1}i_{n}}=1, \; n\in \mathbb{N} \}
$$
and 
$$
\sigma_{A}(w) =i_{1}\cdots i_{n-1}i_{n}\cdots.
$$
By abuse of notation, we use $\Sigma_{A}$ to mean a horseshoe-type forward $f$-invariant subset $X\subset I$.  
Since $A$ is a nonnegative square matrix, from the Perron-Frobenius theorem, it has a maximal nonnegative eigenvalue $e_{A}$. 
If $A$ is also irreducible, $e_{A}$ is simple, unique, and positive.
The statement that maximal Lyapunov exponent of $\sigma_{A}$ is greater than $0$ is equivalent to saying that $e_{A}>1$.  

A point $p\in X$ is called a periodic point of period $n\geq 1$ if $f^{n} (p)=p$. The minimal such $n$ is called the primitive period. 
Following the proof of Sharkovskii's theorem (for example, the reader may refer to book~\cite{D} or book~\cite[Theorem 4.55 in \S4.5]{R}), 
we see that if $f$ has a periodic point whose period is different from all powers of $2$, 
then $f$ has a horseshoe-type invariant subset. For example, if $f$ has a periodic point of period $3$, 
then $f$ has a horseshoe-type invariant subset $\sigma_{A}: \Sigma_{A}\to \Sigma_{A}$ for 
$$
A=\left( \begin{matrix}
                0&1\\
                1&1
     \end{matrix}\right)
     \hbox{ or }
     \left( \begin{matrix}
                1&1\\
                1&0
     \end{matrix}\right)  
                 $$  
with the maximal eigenvalue $e_{A}=(1+\sqrt{5})/2$. 
Thus we have that 

\medskip
\begin{lemma}~\label{finite}
Suppose $f$ is an interval map with zero topological entropy. Then the period of any periodic point of $f$ is $2^{n}$ for some $n\geq 0$.
\end{lemma} 

For any $x\in I$, let
$$
\omega (x) = \cap_{n=0}^{\infty} \overline{\{ f^{k}x \;|\; k\geq n\}}.
$$
be the $\omega$-limit set of $x$ under $f$.  
Let $K=\omega (x)$. Consider the sub-flow 
\begin{equation}~\label{subflow}
\mathcal{K} =\{ (f|K)^{n}: K \to K\}_{n\in \{0\}\cup \mathbb{N}}.
\end{equation}
{\em It is a minimal flow. }

We have two cases: 
\begin{itemize}
\item[(I)]  the set $K$ is finite and 
\item[(II)] the set $K$ is infinite. 
\end{itemize}

Case (I) is easier. In this case, 
$$
K =\{ p_{0}, p_{1}, \cdots, p_{2^{n}-1}\}
$$
is a periodic orbit of period $2^{n}$ for some $n\geq 0$. That is, $f(p_{i}) =p_{i+1 \pmod{2^{n}}}$ for $i=0, 1, \cdots, 2^{n}-1$. 
Therefore, the following lemma is easily verified.  

\medskip
\begin{lemma}~\label{I}
In Case (I), the sub-flow (\ref{subflow}) is equicontinuous, thus, MLS. 
Moreover, $x$ is mean-attractable to $K$.
\end{lemma} 

Case (II) is more complicated. Since the topological entropy of $f$ is zero,  in this case, $K$ contains 
no periodic point (see, for example,~\cite[Proposition 5.21]{R}).
Since $f$ cannot have any horseshoe-type forward $f$-invariant subset in $I$, 
we can find a sequence of sets 
$$
\eta_{n}=\{ I_{n,k}\;|\; 0\leq k\leq 2^{n}-1\},\;\; n\in\mathbb{N},
$$
of closed intervals $I_{n,k}$ such that 
\begin{itemize}
\item[1)] for each $n\geq 1$, intervals in $\eta_{n}$ are pairwise disjoint;
\item[2)] $f(I_{n, k}) =I_{n, k+1 \pmod{2^{n}}}$ for any $n\geq 1$ and $0\leq k\leq 2^{n}-1$;
\item[3)] for each $n\geq 1$ and $I_{n, k}\in \eta_{n}$, two and only two intervals $I_{n+1, k}$ and $I_{n+1, k+2^{n}}$ in $\eta_{n+1}$ are sub-intervals of $I_{n, k}$, that is,
$$
I_{n+1, k}\cup I_{n+1, k+2^{n}}\subseteq I_{n, k};
$$
\item[4)] $K \subseteq \cap_{n=1}^{\infty} \cup_{k=0}^{2^{n}-1} I_{n,k}$.
\end{itemize}
This is a well-known fact about interval maps with zero topological entropy among all experts 
in this field (for example, refer to~\cite[Proposition 5.22]{R}).  
 
We put a symbolic coding on each interval $I_{n,k}$ in $\eta_{n}$ for $n\geq 1$ and $0\leq k\leq 2^{n}-1$.
Consider the binary expansion of $k$ as
$$
k= i_{0} + i_{1} 2+ \cdots i_{m}2^{m} +\cdots +i_{n-1}2^{n-1}
$$
where $i_{0}, i_{1}, \cdots, i_{m}, \cdots, i_{n-1} =0$ or $1$. 
Let us only remember the coding $w_{n}=i_{0}i_{1}\cdots i_{m}\cdots i_{n-1}$ and label $I_{n, k}$ by $w_{n}$, that is, $I_{n, k}=I_{w_{n}}$. 

Define the space of codings
$$
\Sigma_{n} =\{ w_{n} =i_{0}i_{1}\cdots i_{m}\cdots i_{n-1}\;|\; i_{m}\in \{0,1\},\; m=0, 1, \cdots, n-1\}
$$
with the product topology.  
We have a natural shift map 
$$
\sigma_{n}(w_{n}) =i_{1}\cdots i_{m}\cdots i_{n-1}: \Sigma_{n}\to \Sigma_{n-1}.
$$
It is a continuous $2$-to-$1$ map. Then we have an inverse limit system 
$$
\{ (\sigma_{n}: \Sigma_{n}\to \Sigma_{n-1})\;; \; n\in \mathbb{N}\}.
$$
Let 
$$
(\sigma: \Sigma \to \Sigma) =\lim_{\longleftarrow} \;(\sigma_{n}: \Sigma_{n}\to \Sigma_{n-1}).
$$
Explicitly, we have that
$$
\Sigma=\{ w =i_{0}i_{1}\cdots i_{n-1}\cdots \;|\; i_{n-1}\in \{0,1\},\; n\in \mathbb{N} \}
$$
and
$$
\sigma: w=i_{0}i_{1}\cdots i_{n-1}\cdots \to \sigma(w) =i_{1}\cdots i_{n-1}\cdots. 
$$

Consider the metric $d(\cdot, \cdot)$ on $\Sigma$ defined as 
$$
d (w, w') =\sum_{n=1}^{\infty} \frac{|i_{n-1}-i_{n-1}'|}{2^{n}}
$$
for $w=i_{0}i_{1}\cdots i_{n-1}\cdots$ and $w'=i_{0}'i_{1}'\cdots i_{n-1}'\cdots$. 
It induces the same topology as the topology from the inverse limit and makes $\Sigma$ a compact metric space. 

On $\Sigma$, we have a map called the {\em adding machine} denoted as $add$ and defined as follows: for any $w=i_{0}i_{1}\cdots \in \Sigma$, consider the formal binary expansion 
$a=\sum_{n=1}^{\infty} i_{n-1}2^{n-1}$, then $a+1= \sum_{n=1}^{\infty} i_{n-1}'2^{n-1}$ has a unique formal binary expansion. 
Here ``a+1'' is just a notation. It means if $i_{0}=0$, then $i_{0}'=1$ and all other $i_{n}'=i_{n}$ and if $i_{0}=1$, then $i_{0}'=0$ and then consider $i_{1}+1$, and so on.
The adding machine is 
$$
add (w) =i_{0}'i_{1}'\cdots i_{n-1}'\cdots: \Sigma\to \Sigma.
$$
It is a homeomorphism of $\Sigma$ to itself. Moreover, we have

\medskip
\begin{lemma}~\label{sca}
The flow 
$$
\mathcal{A}=\{ add^{n}: \Sigma\to \Sigma\}_{n\in \{0\}\cup \mathbb{N}}
$$ 
is minimal and equicontinuous.
\end{lemma}

We leave the proof for the reader since it is not difficult.  
 
For a point $w=i_{0}i_{1}\cdots i_{n-1}\cdots \in \Sigma$, let $w_{n}=i_{0}i_{1}\cdots i_{n-1}$. 
We use 
$$
[w]_{n}=\{ w'\in \Sigma \;|\; w_{n}'=w_{n}\}
$$
to denote the $n$-cylinder containing $w$. 

One can check that for every $w\in \Sigma$, 
$$
\cdots \subset I_{w_{n}} \subset I_{w_{n-1}}\subset \cdots \subset I_{w_{2}}\subset I_{w_{1}},
$$
is a nested sequence of closed intervals. Therefore
$$
I_{w} =\cap_{n=1}^{\infty} I_{w_{n}}
$$
is a non-empty, connected, and compact subset of $I$. Moreover, we have that $f(I_{w}) =I_{add(w)}$ and
$$
K \subset \cup_{w\in \Sigma} I_{w}.
$$ 

The set $\mathcal{C}=\{ [w]_{n} \;| \; w\in \Sigma, n\in \mathbb{N}\}$ of all $n$-cylinders forms a topological basis for $\Sigma$. Let $\mathcal{B}$ be the $\sigma$-algebra generated $\mathcal{C}$. We have a standard probability measure $\mu$ defined on $\mathcal{B}$ such that $\mu ([w]_{n}) =1/2^{n}$ for all $w\in \Sigma$. We say $\mu$ is $add$-invariant if $\mu (add^{-1}(A))=\mu(A)$ for all $A\in\mathcal{B}$. We say $\mu$ is ergodic if $add^{-1}(A))=A$ for $A\in \mathcal{B}$ implies either $\mu (A) =0$ or $1$. The following lemma is also well-known among experts in this field and not difficult to be verified. We also leave it to the reader. 

\medskip
\begin{lemma}~\label{erg}
The measure $\mu$ is $add$-invariant and ergodic.
\end{lemma} 

In Case (II), we consider two different cases:
\begin{itemize}
\item[(a)] there is a point $w_{0}\in \Sigma$ such that $I_{w_{0}}=\{x_{w_{0}}\}$ contains only one point or 
\item[(b)] all $I_{w}$ for $w\in \Sigma$ are closed intervals.
\end{itemize}

In Case $(II)(a)$, we have that $\omega (x)=\omega (x_{w_{0}})$.   Without loss of generality, we assume $x=x_{w_{0}}$ and every $I_{w}=\{x_{w}\}$ contains only one point. Let 
$$
\tau_{n}=\max_{w_{n}\in \Sigma_{n}} |I_{w_{n}}|.
$$ 
We claim that $\tau_{n}\to 0$ as $n\to \infty$. We prove it by contradiction. 
If there is a subsequence $\{n_{k}\}_{k=1}^{\infty}$ 
of the natural numbers and a number $\epsilon_{0}>0$ such that $\tau_{n_{k}}\geq \epsilon_{0}$, then we have 
an interval $I_{w_{n_{k}}}\in \eta_{n_{k}}$ such that $|I_{w_{n_{k}}}| \geq \epsilon_{0}$. Since $\Sigma$ is a compact metric space, 
we have a convergent subsequence of $\{w_{n_{k}}\}_{k=1}^{\infty}$. Suppose $\{w_{n_{k}}\}_{k=1}^{\infty}$ is itself 
a convergent sequence with the limit point $w\in \Sigma$, then the interval $I_{w_{n_{k}}}$ converges to an interval $I_{w}$ as $k\to \infty$. 
The length $|I_{w}|$ is greater than $\epsilon_{0}$. This contradicts the fact $I_{w}$ contains only one point. This proves the claim. 

\medskip
\begin{lemma}~\label{IIa}
In Case $(II) (a)$, the sub-flow $\mathcal{K}$ in (\ref{subflow}) is equicontinuous, thus, MLS.
\end{lemma} 

\begin{proof}
We can set a one-to-one and onto correspondence $h(w)=x_{w}$ from $\Sigma$ to $K$. 
Since $\tau_{n}\to 0$ as $n\to \infty$, for any $\epsilon >0$, there is a $N>0$ such that $\tau_{N}<\epsilon$.  
 For the given $N$, we can find a $\delta>0$ such that for any $w, w'\in \Sigma$ with $d(w, w')<\delta$, then $w_{N}=w_{N}'$. 
 This implies that both points $x_{w}$ and $x_{w'}$ are in the same interval $I_{w_{N}}$. So $|h(w)-h(w')|\leq \tau_{N}<\epsilon$. 
 This proves that $h$ is continuous. Since both $\Sigma$ and $K$ are compact, we have that $h^{-1}$ is also continuous. 
 Thus $h$ is a homeomorphism. Since $f|K = h\circ add \circ h^{-1}$ and since the flow $\mathcal{A}$ is equicontinuous, 
 the sub-flow $\mathcal{K}$ is equicontinuous.
\end{proof}

\medskip
\begin{lemma}~\label{IIama}
In Case $(II) (a)$,  $x$ is mean-attracted to $K$.
\end{lemma} 

\begin{proof}
Since $K \subset \cap_{n=1}^{\infty} \cup_{w_{n}\in \Sigma_{n}} I_{w_{n}}$ 
and since $\tau_{n}\to 0$ as $n\to \infty$, for any $\epsilon >0$, there is an integer $n_{0}>0$ such that $\tau_{n_{0}} <\epsilon$. Then we have an integer $m_{0}>0$ such that 
$f^{m_{0}}x \in I_{w_{n_{0}}}$ for some $w\in \Sigma$. Since $K\cap I_{w_{n_{0}}}\not=\emptyset$ and $f|K =h\circ add\circ h^{-1}$ is a homeomorphism, 
we have $z\in K$ such that $f^{m_{0}}z\in K\cap I_{w_{n_{0}}}$. This implies that 
$$
|f^{m_{0}}x-f^{m_{0}}z|\leq \tau_{n_{0}}<\epsilon.
$$ 

Since $f^{m_{0}+n} (I_{w_{n_{0}}})\in \eta_{n_{0}}$ for all $n\geq 0$, we have that 
$$
|f^{m_{0}+n}x-f^{m_{0}+n}z|\leq \tau_{n_{0}}<\epsilon
$$ 
for all $n\geq 0$. 
This implies that
$$
    \limsup_{N\to\infty} \frac{1}{N} \sum_{n=1}^{N} |f^n x-f^n z| <\epsilon.
$$
We proved that $x$ is mean-attracted to $K$.
\end{proof}

In Case $(II) (b)$, we consider a fixed interval $I_{w_{0}}$. Let 
$$\epsilon_{0}=|I_{w_{0}}|>0. 
$$
The orbit $\{ add^{n} (w_{0})\; |\; n\in\mathbb{Z}\}$ is dense in $\Sigma$. 
Then we can find two points $u, v\in K$ such that the distance $|u-v|$ is arbitrary small and such that  
the interval $I_{uv}=[u,v]$ contains infinitely many positive integers $n$ such that $I_{add^{-n}(w_{0})} \subset I_{uv}$. 
Then we have  
$$
f^{n}(I_{xy})\supset f^{n}(I_{add^{-n}(w_{0})})= I_{w_{0}}.
$$ 
Thus $|f^{n}x-f^{n}y| \geq \epsilon_{0}$ for infinitely many positive integers $n$.
This proves that

\medskip
\begin{lemma}~\label{IIbeq}
In Case (II)(b), the sub-flow $\mathcal{K}$ in (\ref{subflow}) is not equicontinuous.
\end{lemma}

The main work in this paper is to prove the sub-flow $\mathcal{K}$ in (\ref{subflow}) is still MLS in Case (II)(b).  

\medskip
\begin{lemma}~\label{IIb}
In Case $(II) (b)$, the sub-flow (\ref{subflow})  is MLS.
\end{lemma}

\begin{proof}
Since the probability measure $\mu$ is $add$-invariant and ergodic, the Birkhoff ergodic theorem implies that 
$$
\lim_{N\to \infty}  \frac{\#(\{ 1\leq k\leq N  \;|\; add^{k}(w)\in A\})}{N} 
$$
$$
=\lim_{N\to \infty}  \frac{\#(\{ 1\leq k\leq N \;|\; add^{-k}(w)\in A\})}{N} = \mu (A).
$$
for any $A\in \mathcal{B}$ and $\mu$-a.e. $w\in \Sigma$.

For any $\epsilon >0$, since $I_{w}\cap I_{w'}=\emptyset$ for all $w\not= w'\in \Sigma$ 
and since $\sum_{w\in \Sigma} |I_{w}| \leq |I|<\infty$, there is an integer $N_{0}> |I|/\epsilon+1$ and
$0\leq m_{0} \leq  |I|/\epsilon +1$ points, denoted as 
$w^{j}$, $1\leq j\leq m_{0}$, in $\Sigma$ such that for any $w\not=w^{j}\in \Sigma$ for all $1\leq j\leq m_{0}$ and $n\geq N_{0}$, 
$|I_{w_{n}}| < \epsilon$. 

Let $N_{1}>N_{0}$ be an integer such that $\mu ([w]_{N_{1}}) <\epsilon/m_{0}$ for all $w\in \Sigma$. 
We have a number $\delta >0$ such that for any $u, v\in K$ with $|u-v|<\delta$, the interval $[u,v] \subset I_{w^{u,v}_{N_{1}}}$
for some $w^{u,v}\in \Sigma$. For any given $|u-v|<\delta$ and for each $1\leq j\leq m_{0}$, define
$$
E_{j}=\{k\in \mathbb{N}\;|\; add^{-k}([w^{j}]_{N_{1}})= [w^{u,v}]_{N_{1}}\}.
$$ 
From the Birkhoff ergodic theorem (see the formula in the beginning of the proof), we have that 
\begin{equation}~\label{den}
\lim_{N\to \infty} \frac{\#(E_{j}\cap [1,N])}{N} =\mu ([w^{u,v}]_{N_{1}}) < \frac{\epsilon}{m_{0}}.
\end{equation}
Let $E=\cup_{j=1}^{m_{0}} E_{j}$. Then we have that 
\begin{equation}~\label{dent1}
\lim_{N\to \infty} \frac{\#(E\cap [1,N])}{N}  \leq \lim_{N\to \infty} \sum_{j=1}^{m_{0}} \frac{\#(E_{j}\cap [1,N])}{N}  < \epsilon.
\end{equation}
Thus we have that the uppe density $\overline{D}(E)$ is less than $\epsilon$.

For any $n\in \mathbb{N}\setminus E$, $add^{n} ([w^{u,v}]_{N_{1}}) \not=  [w^{j}]_{N_{1}}$ for all $1\leq j\leq m_{0}$. 
This implies $f^{n}(u), f^{n}(v)\in I_{w_{N_{1}}}$ for $w\not= w^{j}$ for all $1\leq j\leq m_{0}$. Thus, $|f^{n}(u)-f^{n}(v)| <\epsilon$.  
This completes the proof of the lemma. 
 \end{proof}

\medskip
\begin{lemma}~\label{IIbma}
In Case $(II) (b)$,  $x$ is mean-attracted to $K$.
\end{lemma} 

\begin{proof}
We use the same notation as that in the proof of Lemma~\ref{IIb}.
For any $\epsilon >0$, since $I_{w}\cap I_{w'}=\emptyset$ for all $w\not= w'\in \Sigma$ 
and since $\sum_{w\in \Sigma} |I_{w}| \leq |I|<\infty$, there is an integer $N_{0}> |I|/\epsilon+1$ and
$0\leq m_{0} \leq  |I|/\epsilon +1$ points, denoted as 
$w^{j}$, $1\leq j\leq m_{0}$, in $\Sigma$ such that for any $w\not=w^{j}\in \Sigma$ for all $1\leq j\leq m_{0}$ and $n\geq N_{0}$, 
$|I_{w_{n}}| < \epsilon/3$. 

Let $N_{1}>N_{0}$ be an integer such that $\mu ([w]_{N_{1}}) <\epsilon/(3|I|m_{0})$ for all $w\in \Sigma$. 
Fix an interval $I_{w_{N_{1}}'}$ with length less than $\epsilon$. There is an integer $k\geq 0$ such that $f^{k}x\in I_{w_{N_{1}}'}$. Take 
$z_{0}\in I_{w_{N_{1}}'}\cap K$. Then we have a point $z=z_{x, \epsilon}\in I_{(add^{-k}(w'))_{N_{1}}}\cap K$ such that $f^{k}z=z_{0}$. Thus we have that $|f^kx-f^kz|<\epsilon$.

Define
$$
E_{j}=\{k\in \mathbb{N}\;|\; add^{-k}([w^{j}]_{N_{1}})= [w']_{N_{1}}\}.
$$ 
From the Birkhoff ergodic theorem (see the formula in the beginning of the proof of Lemma~\ref{IIb}), we have that 
\begin{equation}~\label{den}
\lim_{N\to \infty} \frac{\#(E_{j}\cap [1,N])}{N} =\mu ([w']_{N_{1}}) < \frac{\epsilon}{3|I|m_{0}}.
\end{equation}
Let $E=\cup_{j=1}^{m_{0}} E_{j}$. Then we have that 
\begin{equation}~\label{dent}
\lim_{N\to \infty} \frac{\#(E\cap [1,N])}{N}  \leq \lim_{N\to \infty} \sum_{j=1}^{m_{0}} \frac{\#(E_{j}\cap [1,N])}{N}  <\frac{\epsilon}{3|I|}.
\end{equation}

Consider the sum
$$
\frac{1}{N} \sum_{n=1}^{N} |f^{n}x-f^{n}z| = \frac{1}{N} \sum_{n=1}^{k} |f^{n}x-f^{n}z| 
$$
$$+ \frac{1}{N} \sum_{n\in \mathbb{N}, k+n\in \mathbb{N}\setminus E} |f^{k+n}x-f^{k+n}z| +\frac{1}{N} \sum_{n\in \mathbb{N}, k+n\in E} |f^{k+n}x-f^{k+n}z|
$$
$$
=I+II+III.
$$

First for $N$ large, we have $I<\epsilon/3$.

Secondly, for any $n\in \mathbb{N}$ such that $k+n\in \mathbb{N}\setminus E$, 
$$
[w]_{N_{1}}=add^{k+n} ([w']_{N_{1}}) \not= [w^{j}]_{N_{1}}.
$$ 
This implies 
$f^{k+n}x, f^{k+n}z\in I_{w_{N_{1}}}$ for $w\not= w^{j}$ for all $1\leq j\leq m_{0}$. Thus, 
$$
|f^{k+n}x-f^{k+n}z| <\frac{\epsilon}{3}, \;\; \forall \; n\in \mathbb{N} \hbox{ and } k+n\in \mathbb{N}\setminus E.
$$
This implies that $II<\epsilon/3$. 

Finally, since $|f^{k+n}x-f^{k+n}z|\leq |I|$ is always true, we have  
$$
III \leq \frac{\# (E\cap [1, N])|I|}{N}.
$$ 
The upper density of $E$ given in (\ref{dent}) implies that 
$$
\limsup_{N\to \infty} III < \frac{\epsilon}{3}.
$$ 
Thus we get eventually 
$$
\limsup_{N\to \infty} \frac{1}{N} \sum_{n=1}^{N} |f^{n}x-f^{n}z| <\epsilon.
$$
This proves that $x$ is mean-attracted to $K$.
\end{proof}

Lemmas~\ref{I}, ~\ref{IIa},~\ref{IIama},~\ref{IIb}, and~\ref{IIbma} complete the proof of Theorem~\ref{mainthm}.

 \end{document}